\documentclass[11pt,a4paper]{article}
\usepackage{amssymb,amsmath,amsthm}
\usepackage{txfonts}
\usepackage{extarrows}
\usepackage{array}
\usepackage{booktabs,enumerate}
\usepackage{multirow}
\usepackage{cite}
\usepackage{geometry}
\usepackage{layouts}
\usepackage[colorlinks=true,urlcolor=blue,
citecolor=red,linkcolor=blue,linktocpage,pdfpagelabels]{hyperref}

\newcommand{\dx}{\mathrm{d}x}
\newcommand{\R}{\mathbb{R}}

\newtheorem{theorem}{Theorem}[section]
\newtheorem{lemma}[theorem]{Lemma}
\newtheorem{proposition}[theorem]{Proposition}

\newtheorem{remark}{Remark}[section]
\theoremstyle{definition}

\numberwithin{equation}{section}

\begin{document}
  \title{\LARGE\bf{Positive solutions for a coupled nonlinear\\ Kirchhoff-type system with vanishing potentials}}
\date{}
\author{Lingzheng Kong,
Haibo Chen\thanks{
	Corresponding author. E-mail: \href{mailto:math_klz@csu.edu.cn}{math\_klz@csu.edu.cn(L. Kong)}, \href{mailto:math_chb@163.com}{math\_chb@163.com(H. Chen)}}\\
{\small School of Mathematics and Statistics, HNP-LAMA, Central South University,}\\
{\small Changsha, Hunan 410083, P.R. China}}
\maketitle
\begin{center}
 \begin{minipage}{14cm}
\begin{abstract}
In this paper, we consider the strongly coupled nonlinear Kirchhoff-type system with vanshing potentials:
\[\begin{cases}
-\left(a_1+b_1\int_{\mathbb{R}^3}|\nabla u|^2\dx\right)\Delta u+\lambda V(x)u=\frac{\alpha}{\alpha+\beta}|u|^{\alpha-2}u|v|^{\beta},&x\in\mathbb{R}^3,\\
-\left(a_2+b_2\int_{\mathbb{R}^3}|\nabla v|^2\dx\right)\Delta v+\lambda W(x)v=\frac{\beta}{\alpha+\beta}|u|^{\alpha}|v|^{\beta-2}v,&x\in\mathbb{R}^3,\\
u,v\in \mathcal{D}^{1,2}(\R^3),
\end{cases}\]
where $a_i>0$ are constants, $\lambda,b_i>0$ are  parameters for $i=1,2$, $\alpha,\beta>1$ and $\alpha+\beta\leqslant 4$, $V(x)$, $W(x)$ are nonnegative continuous potentials, the nonlinear term $F(x,u,v)=|u|^\alpha|v|^\beta$ is not 4-superlinear at infinity. Such problem cannot be studied directly by standard variational methods, even by restricting the associated energy functional on the Nehari manifold, because Palais-Smale sequences may not be bounded. Combining  some new detailed estimates with truncation technique, we obtain the existence of positive vector solutions for the above system when $b_1+b_2$ small and $\lambda$ large. Moreover, the asymptotic behavior of these vector solutions is also explored as $\textbf{b}=(b_1,b_2)\to \bf{0}$ and $\lambda\to\infty$. In particular, our results extend some known ones in previous papers that only deals with the case where $4<\alpha+\beta<6$.

\end{abstract}
\end{minipage}
\end{center}
\vskip6mm
{\section{Introduction and statement of  results}\label{sec1}}
\setcounter{equation}{0}
\par
In this paper,  we study the existence and asymptotic behavior of positive vector solutions for  the following coupled Kirchhoff-type system in $\R^3$:
\begin{equation}\label{11}
\begin{cases}
-\left(a_1+b_1\int_{\R^3}|\nabla u|^2\dx\right)\Delta u+\lambda V(x)u=\frac{\alpha}{\alpha+\beta}|u|^{\alpha-2}u|v|^{\beta},&x\in\R^3,\\
-\left(a_2+b_2\int_{\R^3}|\nabla v|^2\dx\right)\Delta v+\lambda W(x)v=\frac{\beta}{\alpha+\beta}|u|^{\alpha}|v|^{\beta-2}v,&x\in\R^3,\\
u,v\in \mathcal{D}^{1,2}(\R^3),
\end{cases}\tag{$\mathcal{K}_{\bf{b},\lambda}$}
\end{equation}
where $a_i>0$ are constants, $\lambda,b_i>0$ are  parameters for $i=1,2$, $\alpha>1$, $\beta>1$ satisfy $\alpha+\beta\leqslant 4$. We assume that $V(x)$ and $W(x)$ satisfy the following hypotheses:
\vskip2mm
\begin{enumerate}
\item[$(H_1)$]
	$V(x),W(x)\in C(\R^3,[0,\infty))$. $\Omega_1=\textup{int } V^{-1}(0)$ and $\Omega_2=\textup{int } W^{-1}(0)$ have smooth boundaries with $\overline{\Omega}_1= V^{-1}(0)$, $\overline{\Omega}_2= W^{-1}(0)$, and ${\Omega}_1\cap\Omega_2\neq\emptyset$;
\item[$(H_2)$]
	There exists  $c>0$ such that the set $\mathcal{M}=\{x\in\R^3:V(x)W(x)\leqslant   c^2\} $ has finite positive Lebesgue measure.
\end{enumerate}
\par
This kind of conditions, first proposed by Bartsch-Wang\cite{BartschCPDE1995}, implies that $\lambda V$ and $\lambda W$ represent potential wells whose depths are controlled by $\lambda$, and has attracted the attention of several researchers, see \cite{BartschCCM2001,BartschZAMP2000,SJTJDE2014,FSXJDE2010,AlvesCVPDE2016,zhangduJDE2020,ZLZJDE2013}. $\lambda V$ and $\lambda W$ are called steep potential wells if $\lambda$ is large enough. We note that condition $(H_2)$ implies that $\Omega_1$ and $\Omega_2$ have finite positive Lebesgue measures.

More recently, Sun-Wu \cite{SJTJDE2014} considered the following scalar case of \eqref{11}
\begin{equation}\label{14}
\begin{cases}
-\left(a+b\int_{\R^3}|\nabla u|^2\dx\right)\Delta u+\lambda V(x)u=f(x,u),\hspace{1ex}x\in\R^N,\\
u\in H^1(\R^N),
\end{cases}
\end{equation}
where $N\geqslant 3$ and $f(x,u)$ is the potential $V$ satisfies the following conditions:
\vskip2mm
\begin{enumerate}
	\item[$(V_1)$]
	$V\in C(\R^N,[0,\infty))$, $\Omega=\textup{int }V^{-1}(0)$ is nonempty and has smooth boundary with $\overline{\Omega}=V^{-1}(0)$;
	\item[$(V_2)$]
	there exists $c>0$ such that the set $\{x\in\R^N:V(x)<c\}$ is nonempty and has finite Lebesgue measure.
\end{enumerate}
By assuming different conditions on the nonlinear term $f(x,u)$, they obtained the existence and the nonexistence of nontrivial solutions of \eqref{14} by using variational methods. Moreover, the authors also explored the asymptotic behavior of nontrivial solutions with the help of the Lions vanishing lemma \cite{PLLions1984}.
Very recently, Zhang-Du  \cite{zhangduJDE2020} studied \eqref{14} when $N=3$ and $f(x,u)=|u|^{p-2}u$ with $2<p<4$ and proved the existence and asymptotic behavior of positive solutions by combining the truncation technique and the parameter-dependent compactness condition. For more results about Kirchhoff-type problems like \eqref{14}, we refer the readers to \cite{SJTNA2020,DYBDCDS2018,ZLZJDE2013} and the references therein.
\par
Problem \eqref{14} is related to  the stationary case of the following equation proposed by Kirchhoff \cite{Kir}:
\begin{equation}\label{equ1.2}
u_{t t}-\left(a+b \int_{\Omega}|\nabla u|^{2} \mathrm{d} x\right) \Delta  u=f(x, u),\; (x,t)\in \Omega\times\R^+,
\end{equation}
where $\Omega\subset\R^N$ is a smooth domain, $u$ stands for the displacement, $f\in C(\Omega\times \R,\R)$ is the external force, $a$ is the initial tension and $b$ is related to the intrinsic properties of the string. It can be seen as an extension of the classical D'Alembert’s wave equation, particularly, taking into account the subsequent change in string length during the oscillations. Such a nonlocal problems also arise in biological systems that can be applied to describe the growth and movement depending on the average of itself, for example one species' population density. For more physical background, we refer the readers to \cite{ChipotNA1997,ArosioTAMS1996}.
After J.L. Lions \cite{JLLions1978} proposed a functional analysis approach to the equation \eqref{equ1.2},  the Kirchhoff-type problem have been received much attention, see \cite{ZhangJDE2006,zhangJMAA2006,AlvesCMA2005,SJTEdinburgh2016,DYBJFA2015,FigueiredoARMA2014,ZouJDE2012,liuZAMP2022,liuguoZAMP2015,LZSNA2019,LGBJDE2014,LGBCVPDE2015} and the references therein.
\par
We observe that there exist extensive papers  in the study of the coupled elliptic systems on the whole space, we refer the readers to \cite{LvEJQTDE2014,FSXJDE2010,LPJDE2017,SHXMMA2016,cgfMMA2017,WuJMP2012,zhouCMA2013} and the references therein. Let us state some
known results. When $a_1=a_2=1$, $b_1=b_2=0$, the system \eqref{11} is reduced to the following semilinear Schr\"{o}dinger system in $\R^N$:
\begin{equation}
\begin{cases}
-\Delta u+\lambda V(x)u=\frac{\alpha}{\alpha+\beta}|u|^{\alpha-2}u|v|^{\beta},&x\in\R^N,\\
-\Delta v+\lambda W(x)v=\frac{\beta}{\alpha+\beta}|u|^{\alpha}|v|^{\beta-2}v,&x\in\R^N,\\
u,v\in \mathcal{D}^{1,2}(\R^N),
\end{cases}
\end{equation}
where $N\geqslant 3$, $\lambda>0$ is a parameter, $\alpha, \beta>1$ and $\alpha+\beta<2^*=2N/(N-2)$. $V(x)$ and $W(x)$ satisfy $(H_1)$ and $(H_2)$.
In \cite{FSXJDE2010}, Furtado-Silva-Xavier studied the existence and multiplicity of solutions  when the parameter $\lambda$ is large enough. Moreover, they also studied the asymptotic behavior of these solutions when $\lambda\to\infty$ via the Nehari manifold on the limit system.
Shi-Chen \cite{SHXMMA2016} considered system \eqref{11} when $\lambda=1$ and the nonlinear terms are replaced by $\mu F_u(x,u,v)+|u|^{\tau-2}u$ and $\mu F_v(x,u,v)+|v|^4v$, respectively, where $4<\tau<6$ and $\mu>0$ is a parameter. The authors proved  the existence of ground state solutions when $V(x)=W(x)$ is a continuous and asymptotically periodic potential functional and the nonlinear term $F(x,u,v)$ is 4-superlinear at infinity and satisfies the following monotonicity condition, i.e., 
\vskip2mm
\begin{enumerate}
\item[$(F_1)$]
    $\frac{F(x,u,v)}{|(u,v)|^4}\to+\infty$ uniformly in $x\in\R^3$ as $(u,v)\to+\infty$;
\item[$(F_2)$]
	$u\to \frac{F_u(x,u,v)}{|u|^3}$ and $v\to \frac{F_v(x,u,v)}{|v|^3}$ are nondecreasing on $(-\infty,0)\cup(0,+\infty)$.
\end{enumerate}
\par
L\"{u}-Xiao \cite{LvEJQTDE2014} investigated the following coupled Kirchhoff-type systems 
\begin{equation}
\begin{cases}
-\left(a+b\int_{\R^3}|\nabla u|^2\dx\right)\Delta u+\lambda V(x)u=\frac{2\alpha}{\alpha+\beta}|u|^{\alpha-2}u|v|^{\beta},&x\in\R^3,\\
-\left(a+b\int_{\R^3}|\nabla v|^2\dx\right)\Delta v+\lambda W(x)v=\frac{2\beta}{\alpha+\beta}|u|^{\alpha}|v|^{\beta-2}v,&x\in\R^3,
\end{cases}
\end{equation}
where $\alpha,\beta>2$ with $\alpha+\beta<2^*=6$ and the potentials $V$ and $W$ satisfy the following conditions:
\vskip2mm
\begin{enumerate}
\item[$(H'_1)$]
    $V(x),W(x)\in C(\R^3,[0,\infty))$. $\Omega=\textup{int } V^{-1}(0)=\textup{int } W^{-1}(0)$ is nonempty with smooth boundary and $\overline{\Omega}=V^{-1}(0)=W^{-1}(0)$;
\item[$(H'_2)$]
    There exist  $M_1,M_2>0$ such that the sets $\{x\in\R^3:V(x)\leqslant   M_1\}$ and $\{x\in\R^3:W(x)\leqslant   M_2\}$ have finite positive Lebesgue measures.
\end{enumerate}
By using Nehari manifold and Mountain Pass Theorem, they obtain the existence and multiplicity of nontrivial vector solutions when the parameter $\lambda$ is  large enough, but don't study the asymptotic behavior of nontrivial vector solutions. In \cite{LvEJQTDE2014}, the assumption $4<\alpha+\beta$ is crucial to prove the boundedness of Palais-Smale sequences, which implies the 4-superlinear of the nonlinear term at infinity.
\par
In this paper, we extend the above results from two aspects. Firstly, we consider the Kirchhoff-type system \eqref{11} restriced on the subspace of $\mathcal{D}^{1,2}(\R^3)\times\mathcal{D}^{1,2}(\R^3)$, but not of $H^{1}(\R^3)\times H^{1}(\R^3)$, which prevents us from applying Lions vanishing lemma\cite{PLLions1984}  to study the asymptotic behavior of solutions
and exclude the trivial solutions as e.g. in  \cite{zhangduJDE2020,SJTJDE2014,ZLZJDE2013,SJTZAMP2015,WTFCPAA2016,LZSCMA2016}. 
Secondly, we study the nonlinearity of \eqref{11} satisfying $\alpha,\beta>1$ and $\alpha+\beta\leqslant 4$, which violates the conditions $(F_1)$ and $(F_2)$, and prevents us from using Nehari manifold and fibering methods as e.g. in  \cite{LvEJQTDE2014,SHXMMA2016,RabinowitzZAMP1992,FSXJDE2010}. 
It is natural to ask whether there exists  a positive vector solution for system  \eqref{11} restriced on  $\mathcal{D}^{1,2}(\R^3)\times\mathcal{D}^{1,2}(\R^3)$ when $2<\alpha+\beta\leqslant 4$. 
To the best of our knowledge, it seems that very little has been undertaken on this question in the literature.
\par
Motivated by the works described above, the purpose of this paper is to study the strongly coupled Kirchhoff-type system \eqref{11} in the case where $\alpha,\beta>1$ with $\alpha+\beta\leqslant   4$. More specifically, we shall study the existence of positive vector solutions for \eqref{11} in this case. Moreover, we also explore the asymptotic behavior of these solutions as $\bf{b}\to0$ and $\lambda\to \infty$.
\par

\par
To state our main results, we introduce the following spaces  associated with the potentials $V$ and $W$:\[E_{V}=\left\{u\in \mathcal{D}^{1,2}(\R^3):\int_{\R^3}V(x)u^2\dx<+\infty \right\}\]
and
\[E_{W}=\left\{u\in \mathcal{D}^{1,2}(\R^3):\int_{\R^3}W(x)u^2\dx<+\infty \right\}.\]
Thus the natural space in this paper is the Hilbert space $E=E_{V}\times E_{W}$ with the inner product and norm
\[\langle z,\zeta\rangle=\int_{\R^3}(\nabla u\nabla\varphi+\nabla v\nabla\psi+V(x)u\varphi+W(x)v\psi)\dx,\hspace{2ex} ||z||=\langle z,z\rangle^{1/2}\]
for any $z=(u,v)\in E$ and $\zeta=(\varphi, \psi)\in E$. Throughout this paper, without loss of generality, we always assume that $a_1=a_2=1$ in \eqref{11}. Given $\lambda>0$, we consider the Hilbert space   $E_\lambda=(E,||\cdot||_{\lambda})$ equipped with the inner product and norm 
\[\langle z,\zeta\rangle_\lambda=\int_{\R^3}(\nabla u\nabla\varphi+\nabla v\nabla\psi+\lambda V(x)u\varphi+\lambda W(x)v\psi)\dx,\hspace{2ex} ||z||_\lambda=\langle z,z\rangle_\lambda^{1/2}.\]
It is easy to see that $||z||\leqslant  ||z||_{\lambda}$ for $\lambda\geqslant 1$.
\par
We shall  prove the existence of positive vector solutions of \eqref{11}, to do this we consider $\mathcal{J}_{{\bf{b}},\lambda}:E_\lambda\to\R$ given by 
\begin{equation}\label{16}
\mathcal{J}_{\bf{b},\lambda}(z)=\frac{1}{2}||z||_\lambda^2+\frac{1}{4}(b_1\|\nabla u\|^4_{L^2(\R^3)}+b_2\|\nabla v\|^4_{L^2(\R^3)})-\frac{1}{\alpha+\beta}\int_{ \R^3}|u^+|^\alpha|v^+|^\beta \dx.
\end{equation}
In view of the conditions $(H_1)$ and $(H_2)$, the functional $\mathcal{J}_{\bf{b},\lambda}$ is well defined and of class $C^1$. Moreover, it is standard to see that \eqref{11} is variational and its solutions are the critical points of \eqref{16}.
\par
In our first result we study the existence of positive solutions of \eqref{11}. Here we call $z=(u,v)$ a positive vector function if the functions  $u$ and $v$ are positive a.e. in $\R^3$. We observe that, if $z=(u,v)\in H_0^1(\Omega_1)\times H_0^1(\Omega_2)$ is a nontrivial vector solution of \eqref{19}, then by the condition $(H_1)$, $z$ is also a nontrivial vector solution of \eqref{11} for any $\lambda>0$. Therefore, we are interested in  positive vector solution of \eqref{11}, and obviously it does not lie in $H_0^1(\Omega_1)\times H_0^1(\Omega_2)$. Moreover, it is reasonable to regard $b_1$ and $b_2$ as parameters in \eqref{11}, since, as already mentioned before, $b_1$ and $b_2$ are related to the intrinsic properties of the string.
\vskip4mm
\begin{theorem}\label{the1}
	Suppose that  $\alpha$, $\beta>1$ with $\alpha+\beta\leqslant 4$ and conditions $(H_1)$--$(H_2)$ hold. Then there exist four numbers $\Lambda^*>1$, $b_*>0$ and $\delta$, $T>0$ such that for any $\lambda>\Lambda^*$ and $0<b_1+b_2<b_*$, the system \eqref{11} possesses at least one positive vector solution $z_{\bf{b},\lambda}=(u_{\bf{b},\lambda},  v_{\bf{b},\lambda})\in E_\lambda$  satisfying
	\begin{equation}\label{13}
	\delta\leqslant ||z_{\bf{b},\lambda}||_\lambda\leqslant T,\text{ for any } b_1,b_2 \text{ and } \lambda.
	\end{equation} 
\end{theorem}
\begin{remark}
	\rm{}It is worth mentioning that the conditions $(H'_1)$ and $(H'_2)$ are stronger than $(H_1)$ and $(H_2)$, respectively. They are both two different equivalent forms of the conditions $(V_1)$ and $(V_2)$ if we let $V(x)=W(x)$. In \cite{SJTJDE2014}, 
	the condition $(V_2)$ is crucial to prove the continuity of the work space embedded into the Sobolev space $H^1(\R^3)$, which can be applied directly in \cite{LvEJQTDE2014} under the condition $(H'_2)$, but not in the present paper under the condition $(H_2)$.  Hence, we don't expect to find positive vector solutions for \eqref{11} in $H^1(\R^3)$. 
\end{remark}
\begin{remark}
	\rm{}\eqref{11} has no nontrivial vector solution in $H_0^1(\Omega_1)\times H_0^1(\Omega_2)$ when $b_1+b_2$ is large enough. Indeed, if $z=(u,v)\in H_0^1(\Omega_1)\times H_0^1(\Omega_2)$ is a nontrivial vector solution of \eqref{11}, then similar to the proof of Lemma \ref{L21}, we have 
	\[\begin{aligned}
	0=&\|\nabla u\|_{L^2(\Omega_1)}^2 + \|\nabla v\|_{L^2(\Omega_2)}^2+b_1\|\nabla u\|_{L^2(\Omega_1)}^4+b_2\|\nabla v\|_{L^2(\Omega_2)}^4-\int_{\Omega_1\cap\Omega_2}|u|^\alpha|v|^\beta \dx\\
	\geqslant &\|\nabla u\|_{L^2(\Omega_1)}^2 + \|\nabla v\|_{L^2(\Omega_2)}^2+b_1\|\nabla u\|_{L^2(\Omega_1)}^4+b_2\|\nabla v\|_{L^2(\Omega_2)}^4-C\left(\|\nabla u\|_{L^2(\Omega_1)}^2 + \|\nabla v\|_{L^2(\Omega_2)}^2\right)^{\frac{\alpha+\beta}{2}},\\
	\end{aligned}\]
	where $C=S^{-\frac{\alpha+\beta}{2}}|\Omega_1\cap\Omega_2|^{\frac{6-\alpha-\beta}{6}}>0$ and $|\cdot|$ is the Lebesgue measure. This is a contradiction for $\max\{b_1,b_2\}$ large enough.
\end{remark}
Now we give our main ideas for the proof of Theorem \ref{the1}. To overcome the obstacle of finding bounded Palais-Smale sequence for $\mathcal{J}_{\bf{b},\lambda}$, as in \cite{SJPJDE2012}, we use a cut-off function $\xi\in C^\infty(\R_+,[0,1])$ satisfying
\begin{equation}\label{199}
\begin{cases}
\xi(t)=1,&t\in[0,1],\\
\xi(t)=0,&t\in[2,+\infty),\\
\xi'(t)\leqslant 0,&t\in(0,\infty),\\
||\xi'||_\infty\leqslant 2,
\end{cases}
\end{equation}
and  move to study the following truncated functional $\mathcal{J}_{\bf{b},\lambda}^T$: $E_\lambda\to\R$ defined by
\begin{equation}\label{110}
\mathcal{J}_{\bf{b},\lambda}^T(z)=\frac{1}{2}||z||_\lambda^2+\frac{1}{4}\xi\left(\frac{||z||_\lambda^2}{T^2}\right)(b_1\|\nabla u\|^4_{L^2(\R^3)}+b_2\|\nabla v\|^4_{L^2(\R^3)})-\frac{1}{\alpha+\beta}\int_{ \R^3}|u^+|^\alpha|v^+|^\beta \dx,
\end{equation}
for any $T>0$. Firstly, by the Mountain Pass Theorem without Palais-Smale condition \cite{willem1996}, we can obtain a  Palais-Smale sequence $\{z_n\}=\{(u_n,v_n)\}$ of $\mathcal{J}_{\bf{b},\lambda}^T$ at the mountain pass level $c_{\bf{b},\lambda}^T$ for small $b_1+b_2>0$. Secondly, we get an important upper bound for $c_{\bf{b},\lambda}^T$ that is independent of $T$, $b_1,b_2$ and $\lambda$. By choosing an appropriate $T>0$, we can deduce that, up to a subsequence, $||z_n||_\lambda\leqslant T$ for all $n\in\mathbb{N}$ by restricting $b_1+b_2>0$ small enough, and so $\{z_n\}$ is a bounded Palais-Smale sequence of $\mathcal{J}_{\bf{b},\lambda}$, i.e.,   
\par\noindent
$$\sup\limits_{n\in\mathbb{N}}||z_n||_{\lambda}\leqslant T,\hspace{1ex}\mathcal{J}_{\bf{b},\lambda}(z_n)\to c^{T}_{\bf{b},\lambda}\hspace{1ex}\text{and}\hspace{1ex}||\mathcal{J}_{\bf{b},\lambda}'(z_n)||_{E^*_\lambda}\to 0,\hspace{1ex}\text{as}\hspace{1ex}n\to\infty,$$
where $E^*_\lambda$ is the dual space of $E_\lambda$. Thirdly,  to recover the compactness, inspired by \cite{BartschZAMP2000,zhangduJDE2020}, we establish the parameter-dependent compactness condition to prove $\mathcal{J}_{\bf{b},\lambda}$ satisfies $\textup{(PS)}_{c^{T}_{\bf{b},\lambda}}$ condition for $\lambda>0$ large. 
Finally, in order to apply strong maximum principle to prove the positivity  of the solutions, we give a detailed  estimate of the coupling term in the energy functional $\mathcal{J}_{\bf{b},\lambda}$.
\par
In our next results, we study the asymptotic behavior of the positive vector solutions obtained by Theorem \ref{the1} as $\bf{b}\to\bf{0}$ and $\lambda\to\infty$. 
\begin{theorem}\label{the2}
	Let $z_{\bf{b},\lambda}$ be the positive vector solutions of \eqref{11} obtained by Theorem \ref{the1}. Then for any $b_1+b_2\in (0,b_*)$ fixed, $z_{\bf{b},\lambda}\to z_{\bf{b}}$ in $E$ as $\lambda\to\infty$, where $z_{\bf{b}}=(u_{\bf{b}},v_{\bf{b}})\in H_0^1(\Omega_1)\times H_0^1(\Omega_2)$ is a positive vector solution of
	\begin{equation}\label{19}
	\begin{cases}
	-\left(1+b_1\int_{\Omega_1}|\nabla u|^2\dx\right)\Delta u=\frac{\alpha}{\alpha+\beta}|u|^{\alpha-2}u|v|^{\beta},&x\in\Omega_1,\\
	-\left(1+b_2\int_{\Omega_2}|\nabla v|^2\dx\right)\Delta v=\frac{\beta}{\alpha+\beta}|u|^{\alpha}|v|^{\beta-2}v,&x\in\Omega_2,\\
	u\in H_0^1(\Omega_1),v\in H^1_0(\Omega_2).
	\end{cases}\tag{$\mathcal{K}_{\bf{b},\infty}$}
	\end{equation}
\end{theorem}
\begin{remark}
	\rm{}
	It is worth emphasizing there are two usual methods to explore the asymptotic behavior of solutions, one is with the help of the  Lions vanishing lemma \cite{PLLions1984}, see e.g. \cite{BartschCCM2001,SJTJDE2014,zhangduJDE2020,ZLZJDE2013}, the other one is via the Nehari manifold on the limit system, see \cite{FSXJDE2010}. As already mentioned before, neither the Nehari method nor the Lions vanishing lemma is applicable in our case.  Fortunately, the key inequalities \eqref{46}  and \eqref{38} are obtained in Lemma \ref{L21}, allowing us to overcome this difficulty.
\end{remark}
\begin{theorem}\label{the3}
	Let $z_{\bf{b},\lambda}$ be the positive vector   solutions of \eqref{11} obtained by Theorem \ref{the1}. Then for any $\lambda\in (\Lambda^*, \infty )$ fixed, $z_{\bf{b},\lambda}\to z_\lambda$ in $E_\lambda$ as $\bf{b}\to\bf{0}$, where $z_\lambda\in E_\lambda$ is a positive vector solution of
	\begin{equation}\label{3res}
	\begin{cases}
	-\Delta u+\lambda V(x)u=\frac{\alpha}{\alpha+\beta}|u|^{\alpha-2}u|v|^{\beta},&x\in\R^3,\\
	-\Delta v+\lambda W(x)v=\frac{\beta}{\alpha+\beta}|u|^{\alpha}|v|^{\beta-2}v,&x\in\R^3,\\
	u,v\in \mathcal{D}^{1,2}(\R^3).
	\end{cases}\tag{$\mathcal{K}_{\bf{0},\lambda}$}
	\end{equation} 
\end{theorem}
\begin{theorem}\label{the4}
	Let $z_{\bf{b},\lambda}$ be the positive vector   solutions of \eqref{11} obtained by Theorem \ref{the1}. Then $z_{\bf{b},\lambda}\to z_{\bf{0}}$ in $E$ as  $\bf{b}\to\bf{0}$ and $\lambda\to\infty$, where $z_{\bf{0}}=(u_{\bf{0}},v_{\bf{0}})\in H_0^1(\Omega_1)\times H_0^1(\Omega_2)$ is a positive vector solution of 
	\begin{equation}\label{4res}
	\begin{cases}
	-\Delta u=\frac{\alpha}{\alpha+\beta}|u|^{\alpha-2}u|v|^{\beta},&x\in\Omega_1,\\
	-\Delta v=\frac{\beta}{\alpha+\beta}|u|^{\alpha}|v|^{\beta-2}v,&x\in\Omega_2,\\
	u\in H_0^1(\Omega_1),v\in H^1_0(\Omega_2).
	\end{cases}\tag{$\mathcal{K}_{0,\infty}$}
	\end{equation}
\end{theorem}
\begin{remark}	\rm
	\ 
	\indent
	\begin{enumerate}
		\item[(i)]Let $b_1(\text{or } b_2)>0$ be a small fixed parameter, similar to the proof of Theorem \ref{the2}, we can also obtain the asymptotic behavior of these solutions obtained by Theorem \ref{the1} and  as $b_2(\text{or } b_1)\to0$ and  $\lambda\to\infty$.
		\item[(ii)]Let $b_1(\text{or } b_2)>0$ be a small fixed parameter and $\lambda>0$ be a large fixed-paramete, similar to the proof of Theorem \ref{the3}, we can also obtain the asymptotic behavior of these solutions obtained by Theorem \ref{the1} as $b_2(\text{or } b_1)\to0$.
	\end{enumerate}
\end{remark}
\par
The remainder of this paper is organized as follows. In Section \ref{sec2}, we give some preliminary results which are crucial throughout the paper. we prove Theorem \ref{the1} in Section \ref{sec3}. The proofs of Theorems \ref{the2}, \ref{the3} and \ref{the4} will be given in Section \ref{sec4}.
\par
Throughout the paper, we use the following notations:
\begin{itemize}
	\item 
	$H^1(\R^3)$ is the usual Sobolev space equipped with the inner product abd norm
	\[
	(u,v)_{H^1(\R^3)}=\int_{\R^3}(\nabla u\nabla v+uv)\dx;\hspace{2ex}||u||_{H^1(\R^3)}^2=\int_{ \R^3}(|\nabla u|^2+u^2)\dx.
	\]
	\item
	$\mathcal{D}^{1,2}(\R^3) $ is the completion of $C^{\infty}_0(\R^3)$ with respect to the semi-norm 
	$$||u||_{\mathcal{D}^{1,2}}^2=\|\nabla u\|_2^2=\int_{ \R^3}\|\nabla u\|^2\dx.$$
	\item
	$L^r(\Omega)$, $1\leqslant r\leqslant\infty$, $\Omega\subset\R^3$, denotes a Lebesgue space,  the norm in $L^r(\R^3)$ is denoted by  $\|u\|_{r,\Omega}$,  where $\Omega$ is a proper subset of $\R^3$, by $\|u\|_r$ where $\Omega=\R^3$.
	\item
	$S$ is the best constant for the embedding of  $\mathcal{D}^{1,2}(\R^3)\hookrightarrow$ $ L^6(\R^3)$, i.e., 
	\begin{equation}\label{S}
	\|u\|_6\leqslant S^{-\frac{1}{2}}|\|\nabla u\||_{2}\quad \forall u\in \mathcal{D}^{1,2}(\R^3).
	\end{equation}
	\item
	 For a mensurable function $u$, we denote by $u^+$ and $u^-$ its positive and negative parts respectively, given by
	$u^+=\max\{u,0\},\hspace{1ex} u^-=\max\{-u,0\}.$
	\item
	 $\rightarrow$ and $\rightharpoonup$ denote the strong and weak convergence in the related function space respectively.
	\item
	$o_n(1)$ denotes any quantity which tends to zero when $n\to\infty$.
	\item
	Let $E$ be a real  Banach space and $I\in C^1(E,\R)$. We call that $I$ satisfies the Palais-Smale  condition at $c$ ($\textup{(PS)}_{c}$ condition for short), if any sequence  $\{u_n\}\subset E$ satisfying  $I(u_n)\to c$   and $||I'(u_n)||_{E^*}\to0$ possesses a convergent subsequence.	
	\item $C,C_1,C_2,\dots$ denote various positive constants, which may vary from line to line.				
\end{itemize}

{\section{Preliminary results}\label{sec2}}
\setcounter{equation}{0}
\vskip2mm
In this section, we give some preliminary lemmas for the proof of Theorem \ref{the1}. It is notable that the following lemma allows us to show the (PS) condition of $\mathcal{J}_{{\bf{b}},\lambda}$ in another way.
\begin{lemma}\label{L21}
	Suppose that  $\alpha$, $\beta>1$ with $\alpha+\beta\leqslant 4$ and $(H_1)$--$(H_2)$ hold. Then for any $z=(u,v)\in E_{\lambda}$ and $\lambda>0$, there exists a constant $\hat{c}>0$, independent of $\lambda$, such that 
	\[
		\int_{ \R^3}|u|^\alpha|v|^\beta \dx \leqslant   \hat{c} ||z||^{\alpha+\beta}.
	\]
\end{lemma}
\begin{proof}
	We borrow an idea from \cite[Lemmas 2.1--2.2]{FSXJDE2010}. 
	It follows from  $(H_2)$ that
	\begin{equation}\label{21}
	\begin{aligned}
	\int_{\R^3}|uv|\dx&\leqslant \int_{\mathcal{M}}|uv|\dx+\frac{1}{ c}\int_{{\mathcal{M}}^{c}}( V(x)u^2)^{\frac{1}{2}}( W(x)v^2)^{\frac{1}{2}}\dx\\
	&\leqslant \|u\|_6|v|_6|\mathcal{M}|^{\frac{2}{3}}+\frac{1}{c}
	\left(\int_{{\mathcal{M}}^{c}} V(x)u^2\dx\right)^{\frac{1}{2}}\left(\int_{{\mathcal{M}}^{c}}W(x)v^2\dx\right)^{\frac{1}{2}}\\
	&\leqslant S^{-1}\|\nabla u\|_2\|\nabla v\|_2|\mathcal{M}|^{\frac{2}{3}}+\frac{1}{2 c}\int_{\R^3}V(x)u^2+ W(x)v^2\dx\\
	&\leqslant \frac{1}{2}\max\left\{S^{-1}|\mathcal{M}|^{\frac{2}{3}}, \frac{1}{ c} \right\}||z||^2,
	\end{aligned}
	\end{equation}
	where $|\mathcal{M}|$ denotes the  Lebesgue measure of $\mathcal{M}$, and we have used the H\"{o}lder inequality and \eqref{S}. 
	Taking $r=\frac{6}{8-\alpha-\beta}$, then we can obtain that
	\begin{equation}\label{22}
	\frac{\alpha-1}{6}+\frac{\beta-1}{6}+\frac{1}{r}=1.
	\end{equation}
	Since $1<r\leqslant \frac{3}{2}$, there exists $\theta=\frac{\alpha+\beta-2}{16-2(\alpha+\beta)}\in (0,\frac{1}{4}]$ such that $r=2\theta+1$. We can deduce from H\"{o}lder inequality and \eqref{S} that
	\begin{equation}\label{23}
	\begin{aligned}
	\int_{\R^3}|uv|^r\dx&\leqslant \left(\int_{\R^3}|uv|^3\dx\right)^{\theta}\left(\int_{\R^3}|uv|\dx\right)^{1-\theta}\\
	&\leqslant \left(\frac{1}{2}\left(\|u\|_6^6+|v|^6_6\right) \right)^\theta\left(\int_{\R^3}|uv|\dx\right)^{1-\theta}\\
	&\leqslant \left(\frac{1}{2}S^{-3}\left(\|\nabla u\|_2^6+\|\nabla v\|^6_2\right) \right)^\theta\left(\int_{\R^3}|uv|\dx\right)^{1-\theta}\\
	&\leqslant S^{-3\theta}||z||^{6\theta}\left(\int_{\R^3}|uv|\dx\right)^{1-\theta}.
	\end{aligned}
	\end{equation}
	From \eqref{21}--\eqref{23}, we can derive 
	\begin{equation}\label{24}
	\begin{aligned}
	\int_{\R^3}|u|^\alpha |v|^\beta \dx&=\int_{\R^3}|u|^{\alpha-1} |v|^{\beta-1}|uv| \dx\\
	&\leqslant
	\left(\int_{\R^3}|u|^6\dx\right)^\frac{\alpha-1}{6}\left(\int_{\R^3}|v|^6\dx\right)^\frac{\beta-1}{6}\left(\int_{\R^3}|uv|^r\dx\right)^\frac{1}{r}\\
	&\leqslant S^{-\frac{\alpha+\beta-2}{2}}||z||^{\alpha+\beta-2}\left(\int_{\R^3}|uv|^r\dx\right)^\frac{1}{r}\\
	&\leqslant S^{-\frac{9\theta}{2\theta+1}}||z||^{\alpha+\beta-2+\frac{6\theta}{r}}\left(\int_{\R^3}|uv|\dx\right)^{\frac{1-\theta}{r}}\\
	&\leqslant \left(\frac{1}{2}\max\left\{S^{-1}|\mathcal{M}|^{\frac{2}{3}}, \frac{1}{c} \right\}\right)^{\frac{1-\theta}{2\theta+1}}S^{-\frac{9\theta}{2\theta+1}}||z||^{\alpha+\beta}\\
	&= \hat{c}||z||^{\alpha+\beta},
	\end{aligned}
	\end{equation}
	where $\hat{c}=\left(\frac{1}{2}\max\left\{S^{-1}|\mathcal{M}|^{\frac{2}{3}}, \frac{1}{c} \right\}\right)^{\frac{1-\theta}{2\theta+1}}S^{-\frac{9\theta}{2\theta+1}}$ is independent of $\lambda$. The proof is complete.
\end{proof}
The following lemma shows that the the coupled term $\int_{ \R^3}|u|^\alpha|v|^\beta \dx$ has BL-splitting property, which is an another version of Br\'{e}zis-Lieb lemma \cite{Brezis1983}.
\begin{lemma}\label{BL}\textup{(\cite[Lemma 4.2]{FSXJDE2010})}
	Let $\{z_n\}=\{(u_n,v_n)\}\subset E_\lambda$ be such that $(u_n, v_n)\rightharpoonup(u,v)$ in $E_\lambda$. Then
	$$\lim\limits_{n\to\infty}\int_{ \R^3}(|u_n|^\alpha|v_n|^\beta-|u_n-u|^\alpha|v_n-v|^\beta)\dx=\int_{ \R^3}|u|^\alpha|v|^\beta \dx.$$
\end{lemma}
\par
To overcome the obstacle of finding bounded Palais-Smale sequence for $\mathcal{J}_{\bf{b},\lambda}$, we first define a cut-off function $\xi\in C^\infty(\R_+,[0,1])$ (see \eqref{199} above). For any $T>0$, we consider the truncated functional $\mathcal{J}_{\bf{b},\lambda}^T:E_\lambda\to\R$ 
\[
\mathcal{J}_{\bf{b},\lambda}^T(z)=\frac{1}{2}||z||_\lambda^2+\frac{1}{4}\xi\left(\frac{||z||_\lambda^2}{T^2}\right)\left(b_1\|\nabla u\|_2^4+b_2\|\nabla v\|_2^4\right)-\frac{1}{\alpha+\beta}\int_{ \R^3}|u^+|^\alpha|v^+|^\beta \dx.
\]
It is easy to see that $\mathcal{J}_{\bf{b},\lambda}^T$ is of class $C^1$ with derivative
\begin{equation}\label{26}
\begin{aligned}
\langle (\mathcal{J}_{\bf{b},\lambda}^T)'(z),\zeta\rangle=&
\langle z,\zeta\rangle_\lambda+\frac{1}{2T^2}\xi'\left(\frac{||z||_\lambda^2}{T^2}\right)\left(b_1\|\nabla u\|_2^4+b_2\|\nabla v\|_2^4\right)\langle z,\zeta\rangle_\lambda\\
&+\xi\left(\frac{||z||_\lambda^2}{T^2}\right)\left(b_1\|\nabla u\|_2^2\int_{ \R^3}\nabla u\nabla \varphi \dx +b_2\|\nabla v\|_2^2\int_{ \R^3}\nabla v\nabla\psi \dx \right)\\
&-\frac{\alpha}{\alpha+\beta}\int_{\R^3}|u^+|^{\alpha-2}u^+\varphi|v^+|^\beta \dx-\frac{\beta}{\alpha+\beta}\int_{\R^3}|u^+|^{\alpha}|v^+|^{\beta-2}v^+\psi \dx,
\end{aligned}
\end{equation}
for any $z=(u,v)\in E_\lambda$ and $\zeta=(\varphi, \psi)\in E_\lambda$.
With this penalization, for a properly chosen $T>0$ and $b_1+b_2$ small enough, we are able to find a Palais-Smale sequence $\{z_n\}$ of $\mathcal{J}_{\bf{b},\lambda}^T$ satisfying $||z_n||_\lambda\leqslant T$ and so $\{z_n\}$ is also a Palais-Smale sequence of $\mathcal{J}_{\bf{b},\lambda}$ satisfying $||z_n||_\lambda\leqslant T$.
\par
The following Lemmas \ref{L233}--\ref{L244} imply that $\mathcal{J}_{\bf{b},\lambda}^T$ has the mountain pass geometry.
\begin{lemma}\label{L233}
	Suppose that $\alpha,\beta>1$ with $\alpha+\beta\leqslant 4$ and $(H_1)$--$(H_2)$ hold. Then there exist $\rho, \eta>0$ such that for any $b_1$, $b_2$, $T>0$ and $\lambda\geqslant 1$, there holds
	\[\inf\{\mathcal{J}^T_{\bf{b},\lambda}:z\in E_\lambda\hspace{1ex} \textup{with}\hspace{1ex} ||z||_{\lambda}=\rho\}\geqslant \eta.\]
\end{lemma}	
\begin{proof}
	It follows from Lemma \ref{L21} that
	\[\mathcal{J}_{\bf{b},\lambda}^T(z)\geqslant \frac{1}{2}||z||_\lambda^2-\frac{\hat{c}}{\alpha+\beta}||z||_\lambda^{\alpha+\beta}=||z||_\lambda^2\left(\frac{1}{2}-\frac{\hat{c}}{\alpha+\beta}||z||_\lambda^{\alpha+\beta-2}\right),\]
	where $\hat{c}$ is given in Lemma \ref{L21} and independent of $b_1$, $b_2$, $T$ and $\lambda.$ Consequently, by choosing $\rho=\left(\frac{\alpha+\beta}{4\hat{c}}\right)^{\frac{1}{\alpha+\beta-2}}>0$ and $\eta=\frac{1}{4}\rho^2$, we arrive at the desired result.
\end{proof}
\begin{lemma}\label{L244}
	Suppose that $\alpha$, $\beta>1$ with $\alpha+\beta\leqslant 4$ and $(H_1)$--$(H_2)$ hold. Then there exist $b^*>0$ and $e=(e_0,e_0)\in C_0^\infty(\Omega_1)\times C_0^\infty(\Omega_2)$ such that for any $T$, $\lambda>0$ and $b_1+b_2\in (0,b^*)$, we have $\mathcal{J}_{\bf{b},\lambda}^T(e)<0$ with $||e||_\lambda>\rho$.
\end{lemma}
\begin{proof}
	We define the functional $\mathcal{I}_\lambda:E_\lambda\to\R$ by
	\[\mathcal{I}_\lambda(z)=\frac{1}{2}||z||_\lambda^2-\frac{1}{\alpha+\beta}\int_{ \R^3}|u^+|^\alpha|v^+|^\beta \dx.\]
	Let $e_1=e_2\in C_0^\infty(\Omega_1\cap\Omega_2)$. Then we have $V(x)e_1=W(x)e_1=0$ on $\R^3$ and $||(e_1,e_1)||_\lambda=\sqrt{2}|\nabla e_1|_2$. Hence, by $\alpha+\beta>2$, we can see that
	\[\lim\limits_{t\to\infty}\mathcal{I}_{\lambda}(t(e_1,e_1))=\lim\limits_{t\to\infty}\left(t^2\int\limits_{\Omega_1\cap\Omega_2}|\nabla e_1|^2\dx-\frac{t^{\alpha+\beta}}{\alpha+\beta}\int\limits_{\Omega_1\cap\Omega_2}|e_1|^{\alpha+\beta} \dx\right)=-\infty.\]
	So we further have that there exists $e_0\in C_0^\infty(\Omega_1\cap\Omega_2)$ with $|\nabla e_0|_2>\frac{\rho}{\sqrt{2}}$ and $e_0^+\not\equiv0$, independent of $b_1$, $b_2$, $T$ and $\lambda$, such that $\mathcal{I}_{\lambda}(e)\leqslant -1$. Since
	\[\mathcal{J}_{\bf{b},\lambda}^T(e)=\mathcal{I}_{\lambda}(e)+\frac{b_1+b_2}{4}\xi\left(\frac{||e||_\lambda^2}{T^2}\right)|\nabla e_0|_2^4\leqslant -1+\frac{b_1+b_2}{4}|\nabla e_0|_2^4,\]
	there exists  $b^*>0$, independent of $\lambda$ and $T$, such that $\mathcal{J}_{\bf{b},\lambda}^T(e)<0$ for any $\lambda, T>0$ and $b_1+b_2\in(0,b^*)$. The proof is complete.
\end{proof}
By Lemmas \ref{L233}--\ref{L244} and the Mountain Pass Theorem without (PS) condition \cite{willem1996}, we can obtain that for any $\lambda\geqslant 1$, $T>0$ and $b_1+b_2\in(0,b^*)$, there is a $\textup{(PS)}_{c^{T}_{\bf{b},\lambda}}$ sequence $\{z_n\}=\{(u_n,v_n)\}\subset E_\lambda$ such that
\begin{equation}\label{27}
\mathcal{J}_{\bf{b},\lambda}^T\to c^{T}_{\bf{b},\lambda}\geqslant \eta>0\hspace{1ex}\text{and}\hspace{1ex}||(\mathcal{J}_{\bf{b},\lambda}^T)'(z_n)||_{E_\lambda^*}\to 0,\hspace{1ex}\text{as}\hspace{1ex}n\to\infty,
\end{equation}
where 
\[c^{T}_{\bf{b},\lambda}:=\inf\limits_{\gamma \in \Gamma }\max\limits_{t\in[0,1]} \mathcal{J}_{\bf{b},\lambda}^T(\gamma(t)),\]
and  
\[\Gamma=\{\gamma\in C([0,1],E_\lambda);\gamma(0)=(0,0),\gamma(1)=(e_0,e_0)\}.\]
To find a bounded Palais-Smale sequence,  we have an important upper bound for $c^{T}_{\bf{b},\lambda}$, which is the keystone of the truncation technique.
\begin{lemma}\label{L255}
	Suppose that $\alpha$, $\beta>1$ with $\alpha+\beta\leqslant 4$ and the conditions $(H_1)$--$(H_2)$ hold. Then for for any $\lambda\geqslant 1$, $T>0$ and $b_1+b_2\in(0,b^*)$, there exists $D>0$, independent of $b_1$, $b_2$, $T$ and $\lambda$, such tha $c^{T}_{\bf{b},\lambda}\leqslant D$.
\end{lemma}
\begin{proof}
	It follows from $e_0\in C_0^\infty(\Omega_1\cap\Omega_2)$ that
	\[\mathcal{J}_{\bf{b},\lambda}^T(t(e_0,e_0))\leqslant t^2\int\limits_{\Omega_1\cap\Omega_2}|\nabla e_0|^2\dx+\frac{b^*t^4}{4}\left(\int\limits_{\Omega_1\cap\Omega_2}|\nabla e_0|^2\dx\right)^2-\frac{t^{\alpha+\beta}}{\alpha+\beta}\int\limits_{\Omega_1\cap\Omega_2}|e_0^+|^{\alpha+\beta} \dx.\]
	Hence, by the definition of $c^{T}_{\bf{b},\lambda}$, there exists a constant $D>0$ such that
	\[c^{T}_{\bf{b},\lambda}\leqslant \max\limits_{t\in[0,1]}\mathcal{J}_{\bf{b},\lambda}^T(t(e_0,e_0))\leqslant D.\]
	The proof is complete.
\end{proof}
\par
The following key lemma shows that for a properly  chosen $T> 0$,  $\{z_n\}$ is, up to a subsequence, a bounded $\textup{(PS)}_{c^{T}_{\bf{b},\lambda}}$ sequence of  $\mathcal{J}_{\bf{b},\lambda}$
satisfying $||z_n||_\lambda\leqslant T$ and \eqref{27}.
\begin{lemma}\label{L26}
	Suppose that $\alpha$, $\beta>1$ with $\alpha+\beta\leqslant 4$, the conditions $(H_1)$--$(H_2)$ hold, $\{z_n\}=\{(u_n,v_n)\}\subset E_\lambda$ is a $\textup{(PS)}_{c^{T}_{\bf{b},\lambda}}$ sequence of  $\mathcal{J}^T_{\bf{b},\lambda}$, and let $T=\sqrt{\frac{2(\alpha+\beta)(D+1)}{\alpha+\beta-2}}$. Then there exists $b_*\in(0,b^*)$ such that for any $\lambda\geqslant 1$ and $b_1+b_2\in(0,b_*)$, up to a subsequence, $||z_n||_\lambda\leqslant T$.
\end{lemma}
\begin{proof}
	Arguing indirectly, suppose that $||z_n||_{\lambda}>T$ up to a subsequence. We next divide the proof into two separate cases.
	\par
	Case 1: $T<||z_n||_{\lambda}\leqslant \sqrt{2}T$. In this case, by \eqref{26}--\eqref{27} and $\xi'(t)\leqslant 0$ for all $t>0$, we have
	\begin{equation}\label{29}
	\begin{aligned}
	c^{T}_{\bf{b},\lambda}=&\lim\limits_{n\to\infty}\left(\mathcal{J}_{\bf{b},\lambda}^T(z_n)-\frac{1}{\alpha+\beta}\langle (\mathcal{J}_{\bf{b},\lambda}^T)'(z_n),z_n\rangle\right)\\
	=&\lim\limits_{n\to\infty}\left((\frac{1}{2}-\frac{1}{\alpha+\beta})||z_n||^2_{\lambda}-(\frac{1}{\alpha+\beta}-\frac{1}{4})\xi\left(\frac{||z||_\lambda^2}{T^2}\right)\left(b_1|\nabla u_n|_2^4+b_2|\nabla v_n|_2^4\right)\right.\\
	&\phantom{=\;\;}
	\left.-\frac{1}{2(\alpha+\beta)T^2}\xi'\left(\frac{||z||_\lambda^2}{T^2}\right)\left(b_1|\nabla u_n|_2^4+b_2|\nabla v_n|_2^4\right)||z_n||^2_{\lambda}\right)\\
	\geqslant& \liminf\limits_{n\to\infty}\left((\frac{1}{2}-\frac{1}{\alpha+\beta})||z_n||^2_{\lambda}-(\frac{1}{\alpha+\beta}-\frac{1}{4})\left(b_1+b_2\right)||z_n||^4_{\lambda}\right)\\
	\geqslant& (D+1)-\frac{4(4-\alpha-\beta)(b_1+b_2)}{(\alpha+\beta-2)^2}(D+1)^2,
	\end{aligned}
	\end{equation}
	which contradicts the fact that  $c^{T}_{\bf{b},\lambda}\leqslant D$ by choosing $b_*$ small enough.
	\par
	Case 2: $||z_n||_{\lambda}>\sqrt{2}T$. In this case, by the similar computation to \eqref{29} and using the fact that
	$\xi\left(\frac{||z||_\lambda^2}{T^2}\right)=\xi'\left(\frac{||z||_\lambda^2}{T^2}\right)=0$, we get
	\begin{equation}
	\begin{aligned}
	c^{T}_{\bf{b},\lambda}=&\lim\limits_{n\to\infty}\left(\mathcal{J}_{\bf{b},\lambda}^T(z_n)-\frac{1}{\alpha+\beta}\langle (\mathcal{J}_{\bf{b},\lambda}^T)'(z_n),z_n\rangle\right)\\
	=&\lim\limits_{n\to\infty}\left((\frac{1}{2}-\frac{1}{\alpha+\beta})||z_n||^2_{\lambda}-(\frac{1}{\alpha+\beta}-\frac{1}{4})\xi\left(\frac{||z||_\lambda^2}{T^2}\right)\left(b_1|\nabla u_n|_2^4+b_2|\nabla v_n|_2^4\right)\notag\right.\\
	&\phantom{=\;\;}
	\left.-\frac{1}{2(\alpha+\beta)T^2}\xi'\left(\frac{||z||_\lambda^2}{T^2}\right)\left(b_1|\nabla u_n|_2^4+b_2|\nabla v_n|_2^4\right)||z_n||^2_{\lambda}\right)\\
	\geqslant & 2(D+1), 
	\end{aligned}
	\end{equation}
	which is again a contradiction. Hence $||z_n||_{\lambda}\leqslant T$ and the proof is complete.
\end{proof}
\vskip2mm
{\section{Existence of positive vector solutions }\label{sec3}}
\setcounter{equation}{0}
\vskip2mm
To find the critical points of $\mathcal{J}_{\bf{b},\lambda}$, we first  establish the following 
parameter-dependent compactness condition for the functional $\mathcal{J}_{\bf{b},\lambda}$.
\vskip2mm
\begin{proposition}\label{L28}
	Suppose that $\alpha,\beta>1$ with $\alpha+\beta\leqslant 4$ and the conditions $(V_1)$--$(V_2)$ hold, and let $T=\sqrt{\frac{2(\alpha+\beta)(D+1)}{\alpha+\beta-2}}$. Then there exists $\Lambda_1> 1$ such that  $\mathcal{J}_{\bf{b},\lambda}$ satisfies $\textup{(PS)}_{c^{T}_{\bf{b},\lambda}}$ condition in $E_\lambda$ for any $b_1+b_2\in(0,b_*)$ and $\lambda\in(\Lambda_1,\infty)$.
\end{proposition}
\begin{proof}
	Let $\{z_n\}=\{(u_n,v_n)\}\subset E$ be a $\textup{(PS)}_{c^{T}_{\bf{b},\lambda}}$ sequence for $\mathcal{J}_{\bf{b},\lambda}$. From Lemma \ref{L26}, we know that $\{z_n\}$ is bounded in $E_\lambda$, i.e.,
	\[||z_n||_\lambda\leqslant T=\sqrt{\frac{2(\alpha+\beta)(D+1)}{\alpha+\beta-2}}.\]
	Then there exist a subsequence $\{z_n\}$ and $z=(u,v)$ in $E_\lambda$ such that
	\begin{equation}\label{28}
	\begin{cases}
	(u_n,v_n)\rightharpoonup(u,v),& \hspace{1ex}\textup{in} \hspace{1ex}E_\lambda;\\
	(u_n,v_n)\to(u,v),& \hspace{1ex}\textup{a.e. in} \hspace{1ex}\R^3;\\
	(u_n,v_n)\to(u,v),& \hspace{1ex}\textup{in} \hspace{1ex}L_{loc}^{s_1}(\R^3)\times L_{loc}^{s_2}(\R^3),\hspace{1ex} 2\leqslant s_1, s_2<6.\\
	\end{cases}
	\end{equation}
	Now we prove that $z_n\to z$ in $E_\lambda$. Let $\tilde{z}_n=z_n-z$ with $\tilde{u}_n=u_n-u$ and $\tilde{v}_n=v_n-v$. By \eqref{28}, we have 
	\[||z||_\lambda\leqslant \liminf\limits_{n\to\infty}||z_n||_\lambda\leqslant T,\]
	and then
	\begin{equation}\label{211}
	||\tilde{z}_n||_\lambda=||z_n-z||_\lambda\leqslant 2T.
	\end{equation}
	It follows from $(H_2)$ that
	\[\begin{aligned}
	\int_{\R^3}|\tilde{u}_n\tilde{v}_n|\dx&=\int_{\mathcal{M}^c}|\tilde{u}_n\tilde{v}_n|\dx+\int_{\mathcal{M}}|\tilde{u}_n\tilde{v}_n|\dx\\
	&\leqslant \frac{1}{2\lambda c}\int_{\mathcal{M}^c}\lambda V(x)\tilde{u}_n^2+\lambda W(x)\tilde{v}_n^2\dx+ |\tilde{u}_n|_{5,\mathcal{M}}|\tilde{v}_n|_{5,\mathcal{M}}|\mathcal{M}|^{\frac{3}{5}}\\
	&\leqslant \frac{1}{2\lambda c}||\tilde{z}_n||_\lambda^2+o_n(1),
	\end{aligned}\]
	this together with Lemma \ref{L21} implies that
	\begin{equation}\label{2133}
	\begin{aligned}
	\int_{\R^3}|\tilde{u}_n|^\alpha|\tilde{v}_n|^\beta \dx
	&\leqslant S^{-\frac{9\theta}{2\theta+1}}||\tilde{z}_n||_\lambda^{\alpha+\beta-2+\frac{6\theta}{2\theta+1}}\left(\int_{\R^3}|\tilde{u}_n\tilde{v}_n|\dx\right)^{\frac{1-\theta}{2\theta+1}}\\
	&\leqslant \left(\frac{1}{2\lambda c}\right)^{\frac{1-\theta}{2\theta+1}}S^{-\frac{9\theta}{2\theta+1}}||\tilde{z}_n||_\lambda^{\alpha+\beta}+o_n(1),
	\end{aligned}
	\end{equation}
	where $\theta=\frac{\alpha+\beta-2}{16-2(\alpha+\beta)}\in (0,\frac{1}{4}]$. 
	Moreover, there exist $A$, $B\in\R$ such that
	\[\int_{ \R^3}|\nabla u_n|^2\dx\to A^2,\hspace{2ex}\int_{ \R^3}|\nabla v_n|^2\dx\to B^2\]
	and 
	\[\int_{ \R^3}\|\nabla u\|^2\dx\leqslant A^2,\hspace{2ex}\int_{ \R^3}\|\nabla v\|^2\dx\leqslant B^2.\]
	As in \cite{FSXJDE2010}, it follows from \eqref{28} and the Lebesgue Dominated Convergence Theorem that 
	\begin{equation}\label{212}
	\lim\limits_{n\to\infty}\int_{ \R^3}|u_n|^{\alpha-2}u_n\varphi|v_n|^\beta \dx=\int_{ \R^3}|u|^{\alpha-2}u\varphi|v|^\beta \dx,\textup{ for any } \varphi\in C_0^\infty(\R^3)
	\end{equation}
	and
	\begin{equation}\label{213}
	\lim\limits_{n\to\infty}\int_{ \R^3}|u_n|^{\alpha}|v_n|^{\beta-2}v_n\psi \dx=\int_{ \R^3}|u|^{\alpha}|v|^{\beta-2}v\psi \dx,\textup{ for any } \psi\in C_0^\infty(\R^3).
	\end{equation}
	Then by $ \mathcal{J}_{\bf{b},\lambda}'(z_n)\to0$ and \eqref{212}--\eqref{213} , we can get that for any  $\zeta=(\varphi,\psi)\in C_0^\infty(\R^3)\times C_0^\infty(\R^3)$,
	\begin{equation}\label{214}
	\begin{aligned}
	&\langle z,\zeta\rangle_\lambda-\frac{1}{\alpha+\beta}\left(\alpha\int_{\R^3}|u^+|^{\alpha-2}u^+\varphi|v^+|^\beta \dx+{\beta}\int_{\R^3}|u^+|^{\alpha}|v^+|^{\beta-2}v^+\psi \dx\right) \\
	&+b_1A^2\int_{ \R^3}\nabla u\nabla \varphi \dx+b_2B^2\int_{ \R^3}\nabla v\nabla\psi \dx=0.
	\end{aligned}
	\end{equation}
	Taking $\zeta=z$ in \eqref{214}, we have 
	\begin{equation}\label{215}
	\begin{aligned}
	||z||^2_\lambda
	+b_1A^2\|\nabla u\|_2^2+b_2B^2\|\nabla v\|_2^2-\int_{\R^3}|u^+|^{\alpha}|v^+|^\beta \dx =0.
	\end{aligned}
	\end{equation}
	Note that
	\begin{equation}
	\begin{aligned}
	||z_n||^2_\lambda
	+b_1|\nabla u_n|_2^4+b_2|\nabla v_n|_2^4 \dx-\int_{\R^3}|u_n^+|^{\alpha}|v_n^+|^\beta \dx =o_n(1).
	\end{aligned}
	\end{equation}
	By the Br\'{e}zis-Lieb Lemma \cite{Brezis1983}, we have that
	\[
	||\tilde{z}_n||_\lambda^2=||z_n||_\lambda^2-||z||_\lambda^2+o_n(1),
	\]
	\begin{equation}\label{218}
	\|\nabla \tilde{u}_n\|_2^2=\|\nabla u_n\|_2^2-\|\nabla u\|_2^2+o_n(1),\hspace{2ex} \|\nabla \tilde{v}_n\|_2^2=\|\nabla v_n\|_2^2-\|\nabla v\|_2^2+o_n(1).
	\end{equation}
	It follows from \eqref{215}--\eqref{218} and Lemma \ref{BL} that 
	\begin{equation}\label{221}
	\begin{aligned}
	o_n(1)=&||z_n||^2_\lambda
	+b_1|\nabla u_n|_2^4+b_2|\nabla v_n|_2^4 \dx-\int_{\R^3}|u_n^+|^{\alpha}|v_n^+|^\beta \dx \\
	&- ||z||^2_\lambda
	-b_1A^2\|\nabla u\|_2^2-b_2B^2\|\nabla v\|_2^2+\int_{\R^3}|u^+|^{\alpha}|v^+|^\beta \dx \\
	=&||\tilde{z}_n||^2_\lambda +b_1A^2|\nabla \tilde{u}_n|_2^2+b_2B^2|\nabla \tilde{v}_n|_2^2-\int_{\R^3}|\tilde{u}_n^+|^{\alpha}|\tilde{v}_n^+|^\beta \dx.
	\end{aligned}
	\end{equation}
	Then by \eqref{24}, \eqref{211}--\eqref{2133} and \eqref{221}, we get that
	\[
		\begin{aligned}
	o_n(1)=&||\tilde{z}_n||^2_\lambda-\int_{\R^3}|\tilde{u}_n^+|^{\alpha}|\tilde{v}_n^+|^\beta \dx +b_1A^2|\nabla \tilde{u}_n|_2^2+b_2B^2|\nabla \tilde{v}_n|_2^2\\
	\geqslant & ||\tilde{z}_n||^2_\lambda-\int_{\R^3}|\tilde{u}_n|^{\alpha}|\tilde{v}_n|^\beta \dx\\
	=&||\tilde{z}_n||^2_\lambda-\left(\int_{\R^3}|\tilde{u}_n|^{\alpha}|\tilde{v}_n|^\beta \dx\right)^{\frac{\alpha+\beta-2}{\alpha+\beta}}
	\left(\int_{\R^3}|\tilde{u}_n|^{\alpha}|\tilde{v}_n|^\beta \dx\right)^{\frac{2}{\alpha+\beta}}\\
	\geqslant &||\tilde{z}_n||^2_\lambda-(\hat{c})^{\frac{\alpha+\beta-2}{\alpha+\beta}}(2T)^{\alpha+\beta-2}\left(\frac{1}{2\lambda c}\right)^{\frac{1-\theta}{8\theta+1}}S^{-\frac{9\theta}{8\theta+1}}||\tilde{z}_n||^2_\lambda,
	\end{aligned}
	\]
	which implies that there exists $\Lambda_1> 1$ such that $\tilde{z}_n\to 0$ in $E_\lambda$ for all $\lambda>\Lambda_1$ and $0<b_1+b_2<b_*$. Consequently, this completes the proof.
\end{proof}
\begin{lemma}\label{lem3.2}
	Given  $\varepsilon>0$  and  $D>0$ , there exist  $\Lambda_{2}=\Lambda_2\left(\varepsilon, D\right)>0$  and  $R_{\varepsilon}=R\left(\varepsilon, D\right)>0$  such that, if  $\{z_n\}=
	\{(u_{n}, v_{n})\} \subset E_{\lambda}$  is a  $(\mathrm{PS})_{{c^{T}_{\bf{b},\lambda}}}$  sequence for  $\mathcal{J}_{\bf{b},\lambda}$  with  ${c^{T}_{\bf{b},\lambda}} \leqslant D$  and  $\lambda > \Lambda_{2}$ , then
	\[
	\limsup _{n \rightarrow \infty} \int_{B_{R_{\varepsilon}}^{c}}|u_{n}|^{\alpha}|v_{n}|^{\beta}\dx \leqslant \varepsilon.
	\]
\end{lemma}
\begin{proof}
It follows from \eqref{24} and the   boundedness of $\{z_n\}\subset E_\lambda$ that
\begin{equation}\label{equ3.13}
\int_{B_{R}^{c}}|u_n|^\alpha |v_n|^\beta \dx \leqslant C\left(\int_{B_{R}^{c}}|u_nv_n|\dx\right)^{\kappa},
\end{equation}
for any $R > 0$,
where $\kappa =\frac{1-\theta}{2\theta+1}$ for $\theta\in(0,\frac{1}{4}]$. By \eqref{S}, Young and H\" older inequalities, we have
\begin{equation}\label{equ3.14}
\begin{aligned}
\int_{B_{R}^{c} \cap \mathcal{M}}\left|u_{n} v_{n}\right|\dx & \leqslant \frac{1}{2} \int_{B_{R}^{c} \cap \mathcal{M}}\left(\left|u_{n}\right|^{2}+\left|v_{n}\right|^{2}\right)\dx \\
& \leqslant \frac{1}{2} |B_{R}^{c} \cap \mathcal{M}|^{\frac 23}\left(\left\|u_{n}\right\|_{L^{2^{*}}}^{2}+\left\|v_{n}\right\|_{L^{2^{*}}}^{2}\right) \\
& \leqslant C |B_{R}^{c} \cap \mathcal{M}|^{\frac 23}.
\end{aligned}
\end{equation}
On the other hand, from $(H_2)$, we have $V(x)W(x)>  c^2$ in $B_{R}^{c} \cap \mathcal{M}^c$. Then we obtain that
\[
\begin{aligned}
\int_{B_{R}^{c} \cap \mathcal{M}^{c}}\left|u_{n} v_{n}\right|\dx & \leqslant \frac{1}{\lambda c} \int_{B_{R}^{c} \cap \mathcal{M}^{c}} \sqrt{\lambda V(x)}\left|u_{n}\right| \sqrt{\lambda W(x)}\left|v_{n}\right|\dx \\
& \leqslant \frac{1}{2 \lambda c} \int_{\R^3} \left(\lambda V(x) u_{n}^{2}+\lambda W(x) v_{n}^{2}\right)\dx \leqslant \frac{C}{\lambda},
\end{aligned}
\]
this, together with \eqref{equ3.13}--\eqref{equ3.14}, implies that
\begin{equation}\label{equ3.15}
\int_{B_{R}^{c}}|u_n|^\alpha |v_n|^\beta \dx \leqslant C\left(|B_{R}^{c} \cap \mathcal{M}|^{\frac 23}+\frac{1}{\lambda}\right)^{\kappa}.
\end{equation}
Note that $|B_{R}^{c} \cap \mathcal{M}|^{\frac 23}\to0$ as $R\to\infty$. Then the right-hand side of \eqref{equ3.15} is small for $R$ and $\lambda$ large enough. The proof is complete.
\end{proof}
\begin{proof}[\rm{}\textbf{Proof of Theorem \ref{the1}}]
	Let $T=\sqrt{\frac{2(\alpha+\beta)(D+1)}{\alpha+\beta-2}}$.  By Lemmas \ref{L233}--\ref{L244}, we see that $\mathcal{J}_{\bf{b},\lambda}^T$ possesses the Mountain Pass geometry. Using a version of the Mountain Pass Theorem without (PS) condition \cite{willem1996}, we can obtain that for any $\lambda\geqslant 1$, and $0<b_1+b_2<b^*$, $\mathcal{J}_{\bf{b},\lambda}^T$ has a $\textup{(PS)}_{c^{T}_{\bf{b},\lambda}}$ sequence $\{z_n\}=\{(u_n,v_n)\}\subset E_\lambda$. By Lemmas \ref{L255}--\ref{L26}, we can get that there exists $b_*\in(0,b^*)$ such that for any $\lambda\geqslant 1$ and $b_1+b_2\in(0,b_*)$, up to a  subsequence, $\{z_n\}$ is a $\textup{(PS)}_{c^{T}_{\bf{b},\lambda}}$ sequence of $\mathcal{J}_{\bf{b},\lambda}$ satisfying
	\begin{equation}\label{equ3.16}
	\sup\limits_{n\in\mathbb{N}}||z_n||_{\lambda}\leqslant T,\hspace{1ex}\mathcal{J}_{\bf{b},\lambda}(z_n)\to c^{T}_{\bf{b},\lambda}\hspace{1ex}\text{and}\hspace{1ex}||\mathcal{J}_{\bf{b},\lambda}'(z_n)||_{E^*_\lambda}\to 0,\hspace{1ex}\text{as}\hspace{1ex}n\to\infty.
	\end{equation}
	Taking $\Lambda^*=\max\{\Lambda_1,\Lambda_2\}$, it follows from 
    Proposition \ref{L28} that  $\mathcal{J}_{\bf{b},\lambda}$ satisfies $\textup{(PS)}_{c^{T}_{\bf{b},\lambda}}$ condition in $E_\lambda$ for any $b_1+b_2\in(0,b_*)$ and $\lambda\in(\Lambda^*,\infty)$. Thus, we may assume that $z_n\to z_{\bf{b},\lambda}$ in $E_\lambda$, and then
	\[||z_{\bf{b},\lambda}||_{\lambda}\leqslant T,\hspace{1ex}\mathcal{J}_{\bf{b},\lambda}(z_{\bf{b},\lambda})= c^{T}_{\bf{b},\lambda}\geqslant \eta>0\hspace{1ex}\text{and}\hspace{1ex}\mathcal{J}_{\bf{b},\lambda}'(z_{\bf{b},\lambda})= 0.\]
	Since 
	\[\begin{aligned}
	0=&\langle \mathcal{J}_{\bf{b},\lambda}'(z_{\bf{b},\lambda}),(u^-_{\bf{b},\lambda},v^-_{\bf{b},\lambda})\rangle\\
	=&||z^-_{\bf{b},\lambda}||^2_\lambda+b_1\|\nabla u\|^2_2\cdot\|\nabla u^-\|^2_2+b_2\|\nabla v\|^2_2\cdot\|\nabla v^-\|^2_2\\
	\geqslant &  ||z^-_{\bf{b},\lambda}||^2_\lambda,
	\end{aligned}\]
	we get that $u_{\bf{b},\lambda},v_{\bf{b},\lambda}\geqslant 0$ in $\R^3$. Now we claim that $u_{\bf{b},\lambda},v_{\bf{b},\lambda}\not\equiv 0$. Arguing by contradiction, suppose that $z_{\bf{b},\lambda}\equiv 0$, i.e., $u_n,v_n\to0$ in $L^2(B_{R_\epsilon})$. Combining \cite[Lemma 2.1]{FSXJDE2010},  the  boundedness of $\{z_n\}\subset E_\lambda$ with the Young inequality, we deduce that
	\begin{equation}\label{equ3.155}
	\int_{B_{R_{\varepsilon}}}|u_n|^{\alpha}|v_n|^{\beta}\dx \leqslant C_{1}\left(\int_{B_{R_{\varepsilon}}}|u_n v_n|\dx\right)^{\kappa } \leqslant C_{2}\left(\int_{B_{R_{\varepsilon}}}\left(|u_n|^{2}+|v_n|^{2}\right)\dx \right)^{\kappa } \rightarrow 0,
	\end{equation}
as $n\to\infty$.   From \eqref{equ3.16}, we deduce that
\[
\frac{\alpha+\beta-2}{2(\alpha+\beta)}\int_{\R^3}|u_n|^{\alpha}|v_n|^{\beta}\dx\geqslant \frac{\alpha+\beta-2}{2(\alpha+\beta)}\|z_n\|^2_\lambda\geqslant \mathcal{J}_{\bf{b},\lambda}(z_n)-\frac{1}{\alpha+\beta}\langle\mathcal{J}_{\bf{b},\lambda}'(z_n),z_n\rangle=c^{T}_{\bf{b},\lambda}+o_n(1)\|z_n\|_\lambda.
\]
Then, it follows from  Lemma \ref{lem3.2} and \eqref{equ3.155} that, for $\lambda>\Lambda^*$,
\[
\begin{aligned}
\frac{2(\alpha+\beta)}{\alpha+\beta-2}c^{T}_{\bf{b},\lambda}&\leqslant \lim_{n\to\infty} \int_{\R^3}|u_n|^{\alpha}|v_n|^{\beta}\dx\\
&=\lim_{n\to\infty} \int_{B_{R_{\varepsilon}}}|u_n|^{\alpha}|v_n|^{\beta}\dx+\lim_{n\to\infty} \int_{B^c_{R_{\varepsilon}}}|u_n|^{\alpha}|v_n|^{\beta}\dx\leqslant \varepsilon,
\end{aligned}
\]
which contradics the fact that $c^{T}_{\bf{b},\lambda}\geqslant\eta>0$. Hence, the claim holds true.
By using the strong maximum principle in each equation of \eqref{11}, we have that $u_{\bf{b},\lambda},v_{\bf{b},\lambda}> 0$ in $\R^3$. Moreover, note that $z_{\bf{b},\lambda}\neq \bf{0}$ and
	\begin{equation}\label{3133}
	\begin{aligned}
	0=&\langle \mathcal{J}_{\bf{b},\lambda}'(z_{\bf{b},\lambda}),z_{\bf{b},\lambda}\rangle\\
	=&||z_{\bf{b},\lambda}||^2_\lambda
	+b_1|\nabla u_{\bf{b},\lambda}|_2^4+b_2|\nabla v_{\bf{b},\lambda}|_2^4 -\int_{\R^3}|u_{\bf{b},\lambda}|^{\alpha}|u_{\bf{b},\lambda}|^\beta \dx\\
	\geqslant &||z_{\bf{b},\lambda}||^2_\lambda-\hat{c}||z_{\bf{b},\lambda}||^{\alpha+\beta},
	\end{aligned}
	\end{equation}
	then, there exists $\delta >0$, independent of $b_1,b_2$ and $\lambda$, such that $||z_{\bf{b},\lambda}||_\lambda>\delta$ for any $b_1,b_2$ and $\lambda$. The proof is complete.
\end{proof}

{\section{Asymptotic behavior of positive vector solutions }\label{sec4}}
\setcounter{equation}{0}

In this section, we study the asymptotic behavior of positive vector solutions for \eqref{11} and and give the proofs of Theorems \ref{the2}--\ref{the4}.
\par
\vskip2mm
\begin{proof}[\rm{}\textbf{Proof of Theorem \ref{the2}}]
	For any sequence $\lambda_n\to\infty$, let $b_1+b_2\in(0,b_*)$ be fixed and $z_n:=z_{\textbf{b},\lambda_n}$ be the  positive vector  solution of $\mathcal{K}_{\textbf{b},\lambda_n}$ obtained in Theorem \ref{the1}. By \eqref{13}, we can get that for any $n\in\mathbb{N}$,
	\begin{equation}\label{31}
	0<\delta\leqslant ||z_n||_{\lambda_n}\leqslant T.
	\end{equation}
	Then there exist a subsequence $\{z_n\}=\{(u_n,v_n)\}$ and $z_{\bf{b}}=(u_{\bf{b}},v_{\bf{b}})$ in $E$ such that
	\begin{equation}\label{322}
	\begin{cases}
	(u_n,v_n)\rightharpoonup(u_{\bf{b}},v_{\bf{b}}),& \hspace{1ex}\textup{in} \hspace{1ex}E;\\
	(u_n,v_n)\to(u_{\bf{b}},v_{\bf{b}}),& \hspace{1ex}\textup{a.e. in} \hspace{1ex}\R^3;\\
	(u_n,v_n)\to(u_{\bf{b}},v_{\bf{b}}),& \hspace{1ex}\textup{in} \hspace{1ex}L_{loc}^{s_1}(\R^3)\times L_{loc}^{s_2}(\R^3),\hspace{1ex} 2\leqslant s_1, s_2<6.\\
	\end{cases}
	\end{equation}
	We first prove $u_{\bf{b}}=0$ in $\Omega_1^c$. For $m\in\mathbb{N}$, set $C_m=\left\{ x\in\R^3: V(x)\geqslant\frac{1}{m} \right\}$, it follows from \eqref{31} that 
	\[\int_{C_m}u_n^2\dx\leqslant \frac{m}{\lambda_n}\int_{C_m}\lambda_nV(x)u_n^2\dx\leqslant \frac{m}{\lambda_n}||z_n||_{\lambda_n}^2\]
	and thus 
	\[\int_{C_m}u_{\bf{b}}^2\dx\leqslant \liminf\limits_{n\to \infty }\int_{C_m}u_n^2\dx=0,\]
	which implies that $u_{\bf{b}}\equiv 0$ in $\Omega_1^c=\bigcup_{m=1}^{n} C_m$. Since $\Omega_1$ has smooth boundary, we get that $u_{\bf{b}}\in H_0^1(\Omega_1)$. Similarly, we have $v_{\bf{b}}\in H_0^1(\Omega_2)$.
	\par
	For any $\varphi\in C_0^\infty(\Omega_1)$, using $(\varphi,0)$ as a test function, we have
	\begin{equation}
	\left(1+b_1|\nabla u_n|^2_2\right)\int_{\R^3}\nabla u_n\nabla\varphi \dx=\frac{\alpha}{\alpha+\beta}\int_{\R^3}|u_n^+|^{\alpha-2}u_n^+\varphi|v_n^+|^\beta \dx.
	\end{equation}
	Since $\varphi$ has compact support, using the same argument as that in \eqref{212} and taking limit in the above equality, we get
	\begin{equation}\label{34}
	\left(1+b_1|\nabla u_n|^2_2\right)\int\limits_{\Omega_1\cup\Omega_2}\nabla u_{\bf{b}}\nabla\varphi \dx=\frac{\alpha}{\alpha+\beta}\int\limits_{\Omega_1\cup\Omega_2}|u_{\bf{b}}^+|^{\alpha-2}u_{\bf{b}}^+\varphi|v_{\bf{b}}^+|^\beta \dx,\hspace{1ex} \forall \varphi\in C_0^\infty(\Omega_1).
	\end{equation}
	Similarly, for any $\psi\in C_0^\infty(\Omega_2)$, we have 
	\begin{equation}\label{35}
	\left(1+b_2|\nabla v_n|^2_2\right)\int\limits_{\Omega_1\cup\Omega_2}\nabla v_{\bf{b}}\nabla\psi \dx=\frac{\beta}{\alpha+\beta}\int\limits_{\Omega_1\cup\Omega_2}|u_{\bf{b}}^+|^{\alpha}|v_{\bf{b}}^+|^{\beta-2}v_{\bf{b}}^+\psi \dx,\hspace{1ex} \forall \psi\in C_0^\infty(\Omega_2).
	\end{equation}
	\par
	Next we show that $z_n\to z_{\bf{b}}$ in $E$.
	Denote $\tilde{u}_n=u_n-u_{\bf{b}}$ and $\tilde{v}_n=v_n-v_{\bf{b}}$, then $\tilde{z}_n=z_n-z_{\bf{b}}$. Similiar to \eqref{2133}, we have 
	\begin{equation}\label{46}
	\begin{aligned}
	\int_{\R^3}|\tilde{u}_n|^\alpha|\tilde{v}_n|^\beta \dx
	&\leqslant S^{-\frac{9\theta}{2\theta+1}}||\tilde{z}_n||_{\lambda_n}^{\alpha+\beta-2+\frac{6\theta}{2\theta+1}}\left(\int_{\R^3}|\tilde{u}_n\tilde{v}_n|\dx\right)^{\frac{1-\theta}{2\theta+1}}\\
	&\leqslant \left(\frac{1}{2\lambda_n c}\right)^{\frac{1-\theta}{2\theta+1}}S^{-\frac{9\theta}{2\theta+1}}||\tilde{z}_n||_{\lambda_n}^{\alpha+\beta}+o_n(1)=o_n(1),
	\end{aligned}
	\end{equation}
	this together with Lemma \ref{BL} implies that
	\begin{equation}\label{38}
	\lim\limits_{n\to\infty}\int_{\R^3}|u_n^+|^{\alpha}|v_n^+|^{\beta} \dx=\int_{\R^3}|u_{\bf{b}}^+|^{\alpha}|v_{\bf{b}}^+|^{\beta} \dx=\int\limits_{\Omega_1\cup\Omega_2}|u_{\bf{b}}^+|^{\alpha}|v_{\bf{b}}^+|^{\beta} \dx.
	\end{equation}
	By \eqref{212}--\eqref{213} and
	\[\langle \mathcal{J}_{b,\lambda_n}'(z_n),z_n\rangle=\langle \mathcal{J}_{b,\lambda_n}'(z_n),(u_{\bf{b}},0)\rangle=\langle \mathcal{J}_{b,\lambda_n}'(z_n),(0,v_{\bf{b}})\rangle=o_n(1),\]
	we have 
	\begin{equation}
	||z_n||_{\lambda_n}^2+b_1|\nabla u_n|^4_2+b_2|\nabla v_n|^4_2=\int_{\R^3}|u_n^+|^{\alpha}|v_n^+|^{\beta} \dx+o_n(1),
	\end{equation}
	\begin{equation}
	||u_{\bf{b}}||^2+b_1|\nabla u_n|^2_2|\nabla u_{\bf{b}}|^2_2=\frac{\alpha}{\alpha+\beta}\int_{\R^3}|u_{\bf{b}}^+|^{\alpha}|v_{\bf{b}}^+|^{\beta} \dx+o_n(1),
	\end{equation}
	\begin{equation}\label{311}
	||v_{\bf{b}}||^2+b_2|\nabla v_n|^2_2|\nabla v_{\bf{b}}|^2_2=\frac{\beta}{\alpha+\beta}\int_{\R^3}|u_{\bf{b}}^+|^{\alpha}|v_{\bf{b}}^+|^{\beta} \dx+o_n(1).
	\end{equation}
	Up to a subsequence, we may assume that there exists $l_i>0(i=1,2,3)$ such that
	\[||z_n||_{\lambda_n}^2\to l_1,\hspace{1ex}|\nabla u_n|^2_2\to l_2,\hspace{1ex}|\nabla v_n|^2_2\to l_3.\]
	It follows from \eqref{38}--\eqref{311} that
	\[\begin{aligned}
	l_1+b_1l_2^2+b_2l_3^2&=||u_{\bf{b}}||^2+b_1l_2|\nabla u_{\bf{b}}|^2_2+||v_{\bf{b}}||^2+b_2l_3|\nabla v_{\bf{b}}|^2_2\\
	&\leqslant ||z_{\bf{b}}||^2+b_1l_2^2+b_2l_3^2.
	\end{aligned}\]
	Then, $l_1\leqslant ||z_{\bf{b}}||^2$. By the weakly lower semi-continuity of $||\cdot||$, we have
	\[||z_{\bf{b}}||^2\leqslant \liminf\limits_{n\to\infty}||z_n||^2\leqslant \limsup\limits_{n\to\infty}||z_n||^2\leqslant \lim\limits_{n\to\infty}||z_n||^2_{\lambda_n}=l_1\le||z_{\bf{b}}||^2,\]
	which implies that $z_n\to z_{\bf{b}}$ in $E$.
	\par
	Finally, we are going to show that $z_{\bf{b}}$ is a positive vector solution of $\mathcal{K}_{\bf{b},\infty}$. It follows from \eqref{34}--\eqref{35} that for any $\varphi\in C_0^\infty(\Omega_1)$ and $\psi\in C_0^\infty(\Omega_2)$,
	\begin{equation}
	\left(1+b_1\int_{\Omega_1}|\nabla u_{\bf{b}}|^2 \dx\right)\int_{\Omega_1}\nabla u_{\bf{b}}\nabla\varphi \dx=\frac{\alpha}{\alpha+\beta}\int_{\Omega_1}|u_{\bf{b}}^+|^{\alpha-2}u_{\bf{b}}^+\varphi|v_{\bf{b}}^+|^\beta \dx
	\end{equation}
	and
	\begin{equation}
	\left(1+b_2\int_{\Omega_2}|\nabla v_{\bf{b}}|^2 \dx\right)\int_{\Omega_2}\nabla v_{\bf{b}}\nabla\psi \dx=\frac{\beta}{\alpha+\beta}\int_{\Omega_2}|u_{\bf{b}}^+|^{\alpha}|v_{\bf{b}}^+|^{\beta-2}v_{\bf{b}}^+\psi \dx.
	\end{equation}
	Thus, $z_{\bf{b}}=(u_{\bf{b}},v_{\bf{b}})$ is a nonnegative vector solution of $\mathcal{K}_{\bf{b},\infty}$ by the density of $C_0^\infty(\Omega_i)$ in $H_0^1(\Omega_i)(i=1,2)$. By \eqref{31} and $z_n\to z_{\bf{b}}$ in $E$, we have
	\[||z_{\bf{b}}||=\lim\limits_{n\to\infty}||z_n||_{\lambda_n}\geqslant \delta>0,\]
	and so $u_{\bf{b}},v_{\bf{b}}\neq0$. 
	We can deduce that $u_{\bf{b}}>0$ in $\Omega_1$ and $v_{\bf{b}}>0$ in $\Omega_2$ by the strong maximum principle. The proof is complete.
\end{proof}
\begin{proof}[\rm{}\textbf{Proof of Theorem \ref{the3}}]
	For any sequence $b_n\to0$, let $\lambda\in(\Lambda^*,\infty)$ be fixed and $z_n:=z_{\textbf{b}_n,\lambda}$ be the  positive vector  solution of $\mathcal{K}_{\textbf{b}_n,\lambda}$ obtained in Theorem \ref{the1}. By \eqref{13}, we can get that for any $n\in\mathbb{N}$,
	\begin{equation}\label{313}
	0<\delta\leqslant ||z_n||_{\lambda}\leqslant T,
	\end{equation}
	i.e., $\{z_n\}$ is bounded in $E_\lambda$. Then there exist a subsequence $\{z_n\}$ and $z_\lambda=(u_\lambda,v_\lambda)$ in $E_\lambda$ such that $z_n\rightharpoonup z_\lambda$ in $E_\lambda$. By a similar argument as used in the proof of Proposition \ref{L28}, we get that $z_n\to z_\lambda$ in $E_\lambda$.
	\par
	Now, it is enough to show that $z_\lambda$ is a positive vector solution of \eqref{3res}. Note that $\mathcal{J}_{b_n,\lambda}'(z_n)\to0$, then for any $\zeta=(\varphi,\psi)\in E_\lambda$, using $(\varphi,0)$ as a test function, we have
	\begin{equation}\label{314}
	\int_{\R^3}(\nabla u_\lambda\nabla\varphi+\lambda V(x)u_\lambda\varphi) \dx=\frac{\alpha}{\alpha+\beta}\int_{\R^3}|u_\lambda^+|^{\alpha-2}u_\lambda^+\varphi|v_\lambda^+|^\beta \dx.
	\end{equation}
	Similarly, using $(0,\psi)$ as a test function, we have
	\begin{equation}\label{315}
	\int_{\R^3}(\nabla v_\lambda\nabla\psi+\lambda W(x)v_\lambda\psi) \dx=\frac{\beta}{\alpha+\beta}\int_{\R^3}|u_n^+|^{\alpha}|v_n^+|^{\beta-2}v_\lambda^+\psi \dx.
	\end{equation}
	Thus, $z_\lambda=(u_\lambda,v_\lambda)$ is a nonnegative vector solution of \eqref{3res}. By \eqref{313} and $z_n\to z_\lambda$ in $E_\lambda$, we have
	\[||z_\lambda||_\lambda=\lim\limits_{n\to\infty}||z_n||_{\lambda}\geqslant \delta>0,\]
	and so $u_\lambda,v_\lambda\neq0$.
	We have that $u_\lambda,v_\lambda>0$ in $\R^3$ by the strong maximum principle.
\end{proof}
\begin{proof}[\rm{}\textbf{Proof of Theorem \ref{the4}}] The proof of Theorem \ref{the4} is essentially same as the proof of Theorem \ref{the2}, so we omit it here.
\end{proof}

\par\noindent
\textbf{Acknowledgments}
\vskip2mm
This work is  supported by National Natural Science Foundation of China (No. 12071486).

{}
\end{document}